\RequirePackage{fix-cm}
\documentclass[smallextended]{svjour3}       
\usepackage{amsfonts,graphicx}
\usepackage{sriram}
\usepackage{algorithm2e}
\usepackage{bm,mathrsfs}

\usepackage{pifont}
\newcommand{\tick}{\ding{51}}%
\newcommand{\cross}{\ding{55}}%
\usepackage{listings}
\usepackage{xcolor,colortbl}
\definecolor{Gray}{gray}{0.85}
\definecolor{LightCyan}{rgb}{0.88,1,1}
\newcolumntype{a}{>{\columncolor{Gray}}c}
\newcolumntype{b}{>{\columncolor{white}}c}

\numberwithin{equation}{section}

\begin{document}

\title{Bernstein Inequalities for  Constrained Polynomial Optimization Problems.
}

\titlerunning{Bernstein Inequalities for Constrained
  Polynomial Optimization Problems.}  

\author{Mohamed Amin Ben Sassi  \and
        Sriram Sankaranarayanan
}


\institute{Mohamed Amin Ben Sassi \at
              University of Colorado, Boulder , USA. \\
              \email{mohamed.bensassi@colorado.edu}           
           \and
            Sriram Sankaranarayanan\at
           University of Colorado, Boulder , USA. \\
              \email{srirams@Colorado.EDU@colorado.edu}  
}

\date{Received: date / Accepted: date}

\maketitle

\begin{abstract}
  In this paper, we examine linear programming (LP) relaxations based
  on Bernstein polynomials for polynomial optimization problems
  (POPs).  We present a progression of increasingly more precise LP
  relaxations based on expressing the given polynomial in its
  Bernstein form, as a linear combination of Bernstein
  polynomials. The well-known bounds on Bernstein polynomials over the
  unit box combined with linear inter-relationships between Bernstein
  polynomials help us formulate ``Bernstein inequalities'' which yield
  tighter lower bounds for POPs in bounded rectangular domains. The
  results can be easily extended to optimization over polyhedral and
  semi-algebraic domains.  We also examine techniques to increase the
  precision of these relaxations by considering higher degree
  relaxations, and a branch-and-cut scheme.  
\end{abstract}

%
%
\section{Introduction}
In this paper, we examine linear programming relaxations for
polynomial optimization problems (POP) that seek to optimize a
multivariate polynomial $p(\vx)$ over a compact interval domain
$ \vx \in [\ell, u]$. Our approach is based on two ideas: (a) We
consider a reformulation of the problem as a linear program using
Bernstein polynomials. However, doing so also increases the number of
decision variables and constraints in the problem. (b) Next, we
present valid inequalities for improving the approximation.  These
inequalities are derived from well-known properties of Bernstein
polynomials that yield linear inter-relationships between the decision
variables of the linear program. Our approach is extended to handle
compact domains described by semi-algebraic constraints. We present a
branch-and-cut scheme that introduces the cutting plane inequalities
hand-in-hand with a decomposition of the feasible region.

The problem of optimizing polynomials over an interval is well-known
to be non-convex, and is in fact NP-hard. Nevertheless, well-known
classes such as linear, quadratic, or even integer linear programs can
be viewed as particular cases of POPs.  Also, since polynomials
provide a good approximation for non linear functions, solving POPs
efficiently is a big step toward handling more complex problems.
Finally, a lot of problems arising from disparate domains such as
biology, robotics and engineering can be formulated as POPs. Our
interest is motivated by verification and synthesis problems for
dynamical systems such as safety, reachability and stability verification.
These problems can be reduced to POPs. In
fact, the motivation of this paper comes from our previous work, where
we aim to prove stability for polynomial dynamical
systems~\cite{Bensassi+Sriram}.  Therein, Bernstein polynomials were
used as an alternative to the well-known sum of squares (SOS) approach
in order to avoid the numerical issues of semi-definite programming
(SDP)~\cite{ParilloSOS,Lasserre01globaloptimization,Shor/1987/Class}. In
this regard, the Simplex algorithm can be implemented in exact
arithmetic to yield numerically validated lower bounds to the optimal
value of the POP, thus formally establishing the stability of the
process.  The success of the approach in a large set of benchmarks
motivates us to go further, improve the results and make them known in
an optimization context.

More precisely, we show how POPs can be relaxed to linear programs
thanks to the use of Bernstein polynomials, and a well-known
reformulation-linearization technique (RLT) described by Sherali et
al~\cite{sherali91,sherali97}. In fact, the properties of Bernstein
polynomials inside the unit box offer us an elegant approach to improving
the RLT approach. We formulate these properties as linear inequalities
 to obtain guaranteed lower (upper) bounds
for our minimization (maximization) problems. This will be useful in
cases where the POP does not need to be solved exactly. In the latter
case, we combine our inequalities with a branch-and-bound
decomposition process originally described by
Nataraj et al~\cite{Nataraj2007}.

We evaluate our approach using a set of
benchmarks described in Nataraj et al~\cite{Nataraj2007} to
characterize the effect of adding the extra Bernstein inequalities to
the RLT approach. We observe that while the addition of these
inequalities improves the lower bound, it is not sufficient
for yielding tight bounds. Next, we consider the addition of
Bernstein inequalities in a ``branch-and-cut'' approach that combines
the addition of cutting planes ``on-demand'' with a branch-and-bound
decomposition of the domain. We find that all approaches eventually
yield tight bounds on the value of the global optimum. Therefore, we
compare the computational time for various approaches. Finally, we
compare the various approaches on benchmarks from our previous
work~\cite{Bensassi+Sriram} which consists of a set of polynomial
Lyapunov functions used as stability proofs for polynomial dynamical
systems. In this particular case, our goal is to show that the
functions are non-negative over a domain. We adapt the
branch-and-bound scheme for this application to evaluate its
effectiveness.

The results of our evaluation are mixed: we observe that adding
cutting plane inequalities does result in tighter lower bounds on the
optimum and therefore examining fewer cells in the branch-and-bound
approach. However, this comes at the cost of obtaining larger linear
programs due to the extra inequalities, and therefore, an overall
larger computation time. We show that the careful
consideration of inequalities to be introduced yields a ``sweet spot''
for better approximations using less computation time.

\subsection{Organization}
In Section~\ref{sec:preliminaries}, we present basic notions and
properties related to Bernstein polynomials. All the results of this
section are quite standard , therefore proofs are omitted.
Section~\ref{Sec:lp-Bernstein-relaxations} is the core of the
paper. In this section, Bernstein polynomials and their properties
inside the unit box are translated into a series of inequalities
yielding a corresponding set of LP relaxations of increasing
precision.  An iterative approach mixing these relaxations is
presented, and a criterion for checking if the given lower bound meets
the optimal value of the original problem are also given. In
Section~\ref{Sec:improvement}, we show how bounds can be made
arbitrary tighter using some techniques such as decomposition (branch-
and-bound scheme).

\subsection{Related Work}

Since solving a POP is generally NP-hard, existing work consists of
relaxing it in order to obtain an easier problem for which efficient
solvers exist. In the literature, we can distinguish two types of
relaxations. The first class is called LP relaxations. These
approaches approximate the POP using linear programs that can be
efficiently solved using an LP solver. A popular LP relaxation is the
reformulation linearization technique (RLT) given by Sherali et
al~\cite{sherali91,sherali97}. The approach was improved by Nataraj et
al~\cite{Nataraj2007} for solving POPs, wherein the use of the Bernstein basis was
proposed as an improvement.  In particular, Nataraj et al made use of
the property that Bernstein polynomial coefficients over a box form a
lower bound of the polynomial. In this work, we show that this
property is simply the optimal value of a LP formed by a series of
inequalities that relate one Bernstein polynomial to another. In doing
so, we formulate numerous valid inequalities that improve
substantially on this bound. Another recent approach called DSOS
(Diagonally-dominant Sum of Squares) was formulated by Ali Ahmadi et
al~\cite{AliAhmadi+Majumdar/2014/DSOS} by relaxing positive semi-definiteness of a matrix using the
stronger condition of \emph{diagonal
  dominance}. In fact, Ali
Ahmadi's approach can be seen as selecting a finite set of generators
from the infinitely generated cone of positive polynomials in the
polynomial ring $\reals[\vx]$.  In contrast, our approach also adds a
finite set of generators to the cone of positive polynomials over a
compact interval. Naturally, both choices of finite bases involve a
tradeoff that are optimal for certain classes of problems. In
particular, we choose the Bernstein polynomials and utilize the set of
linear inter-relationships between these. Extending our approach to
possibly cover the polynomial basis used in the DSOS approach is
currently under investigation.

As an alternative to LP relaxations, we can formulate SDP relaxations. In 2001, Lassere proposed
what was called a Linear matrix equality (LMI)
relaxation~\cite{Lasserre01globaloptimization}. The main idea is to
map the polynomial optimization problem to an optimization problem
over probability measures and then use results from moment
theory. Subsequently, Parillo introduced the SOS programming approach
that has become one of the most popular SDP
relaxations~\cite{ParilloSOS}.  Theoretically, following the
comparison made by Lasserre~\cite{LasserreSDPvsLP} between SDP (LMI)
and LP (RLT) relaxations, one concludes that the SDP approach is much
more precise at the extra (polynomial) cost of solving an SDP.  In
fact the comparison points out that for the LP (RLT) relaxation,
convergence results to the optimal value are not always guaranteed, in
contrast to SDP relaxations. Also, the comparison shows that RLT
cannot be exact whenever the global optimum belongs to the interior of
the feasible set.  We will show in this paper, that this claim does
not remain true (see Example~\ref{ex2}) when the Bernstein
inequalities suggested here are used. Furthermore, in practice, the
SDP approach suffers from numerical issues. This was pointed in our
previous work~\cite{Bensassi+Sriram} when using SOS programming for
Lyapunov function synthesis.  Other approaches like interval
methods~\cite{Moore+Others/2009/Interval} and decomposition techniques
exists. In this paper, we will focus on the related scheme given by
Nataraj et al in~\cite{Nataraj2007}, since it is fully based on the
use of Bernstein coefficients. We will build on this approach by
adding the extra Bernstein inequalities.


\section{Overview of Bernstein Polynomials}\label{sec:preliminaries}
Bernstein polynomials were first proposed by Bernstein as a
constructive proof of Weierstrass approximation
theorem~\cite{Bernstein/1912/Demonstration}, and are useful in many
engineering design applications for approximating geometric
shapes~\cite{Farouki/2012/Bernstein}. They form a basis for
approximating polynomials over a compact interval, and have nice
properties inside the unit box
(see~\cite{Munoz+Narkavicz/2013/Formalization} for more details). We
first examine Bernstein polynomials and their properties for the
univariate case, and then extend them to multivariate polynomials
(see~\cite{Bernstein1,Bernstein2}).
\begin{definition}[Univariate Bernstein Polynomials]
Given an index \\$i \in \{0,\ldots,m\}$, the $i^{th}$ univariate Bernstein
polynomial of degree $m$ over $[0,1]$ is given by the following
expression:
\begin{equation}
 \beta_{i,m}(x)=\left(
\begin{array}{c}
 m \\ i
\end{array}
\right) x^i(1-x)^{m-i}.
\end{equation}
\end{definition}
Using these polynomials, monomials can be written as follows:
\begin{equation}
\label{chap12:bernbasis}
 x^i=\displaystyle{\sum_{j=i}^{m} \frac{\left(
\begin{array}{c}
 j \\ i
\end{array}
\right)}{\left(
\begin{array}{c}
 m \\ i
\end{array}
\right)} \beta_{j,m}(x)},\text{ for all }i=0,\dots,m.
\end{equation} 

Then, in the Bernstein polynomial basis, polynomial $p(x)=\sum\limits_{j=0}^m p_jx^j$ of
degree $m$ can be written as:
$$
p(x)=\displaystyle{\sum_{i=0}^{m}  b_{i,m} \beta_{i,m}(x) }
$$         
where for all $i=0,\dots,m$:
\begin{equation}\label{eq:bernstein-coeff}
 b_{i,m}=\sum_{j=0}^i \frac{\left(
\begin{array}{c}
 i \\ j
\end{array}
\right) 
}
{\left(
\begin{array}{c}
 m \\ j
\end{array}
\right) } p_j.
\end{equation}
The coefficients $b_{i,m}$ are called the \emph{Bernstein coefficients} of the polynomial $p$.

Bernstein polynomials have many interesting properties on the unit interval $[0,1]$. We summarize the most relevant ones for our applications.
\begin{lemma}\label{lem:bernprop}
Bernstein polynomials 
have the following properties: 
\begin{enumerate}
\item Unit partition: $\displaystyle{\sum_{i=0}^{m} \beta_{i,m}(x) }=1.$
\item Bounds: $0 \le \beta_{i,m}(x) \le \beta_{i,m}(\frac{i}{m}),\; \forall i=0,\dots,m.$
\item Induction: $\beta_{i,m-1}(x)=\frac{m-i}{m}\beta_{i,m}(x)+\frac{i+1}{m}\beta_{i+1,m}(x),\; \forall i=0,\dots,m-1.$
\end{enumerate}
\end{lemma}
Using these properties, the following result holds:
\begin{corollary}\label{cor:bernbound}
On the interval $[0,1]$, the following inequality
 holds~\cite{Garloff93}:
\begin{equation}
 \displaystyle{\min_{i=0,\dots,m}b_{i,m}\le p(x) \le \max_{i=0,\dots,m}b_{i,m}}.
\end{equation}
The equality $ \displaystyle{\min_{i=0,\dots,m}b_{i,m}=\min_{x\in [0,1]} p(x)}$, respectively $ \displaystyle{\max_{i=0,\dots,m}b_{i,m}=\max_{x\in [0,1]} p(x)}$, holds iff
$ \displaystyle{\min_{i=0,\dots,m}b_{i,m}\in\{b_{0,m},b_{m,m}\}}$, respectively $ \displaystyle{\max_{i=0,\dots,m}b_{i,m}\in\{b_{0,m},b_{m,m}\}}$. This is commonly called
\emph{ the vertex condition}. 
\end{corollary}

We generalize the previous notions to the case of multivariate
polynomials i.e $p(x)=p(x_1,\dots,x_n)$ where $x=(x_1,\dots,x_n)\in U=[0,1]^n$. 
For \emph{multi indices},
$I=(i_1,\dots,i_n)\in \mathbb{N}^n$,
$J=(j_1,\dots,j_n)\in \mathbb{N}^n$, we will use the following notation
throughout this paper:
\begin{itemize}
\item $I+J=(i_1+j_1,\dots,i_n+j_n).$
\item ${x}^I={x_1}^{i_1}\times {x_2}^{i_2} \dots \times {x_n}^{i_n}$. 
\item $I \le J \; \iff \; i_l \le j_l,$  for all  $l=1,\dots n.$ 
\item $\frac{I}{J}=\left(\frac{i_1}{j_1},\dots,\frac{i_n}{j_n}\right)$ and 
 $\left(
\begin{array}{c}
 I \\ J
\end{array}
\right)=\left(
\begin{array}{c}
 i_1 \\ j_1
\end{array}
\right)\dots \left(
\begin{array}{c}
 i_n \\ j_n
\end{array}
\right).$
\item $I_{r,k}=(i_1,\dots,i_{r-1},i_r+k,i_{r+1},\dots,i_n)$ where $r\in\{1,\dots,n\}$ and $k\in \mathbb{Z}$.
\end{itemize}
Let us fix our maximal degree $\delta=(\delta_1,\dots,\delta_n)\in \mathbb{N}^n$ for a multivariate polynomial $p$ ($\delta_l$ is the maximal degree of $x_l$ for all $l=1,\dots,n$).
Then the multivariate polynomial $p$ can be written as:
$$
p(x) = \sum_{I\le \delta} p_I x^I \text{ where } p_I \in \mathbb{R},\; \forall I\le \delta.
$$
Multivariate Bernstein polynomials are given by products of the univariate polynomials:
\begin{equation}
 B_{I,\delta}(x) = \beta_{i_1,\delta_1}(x_1) \dots \beta_{i_n,\delta_n}(x_n) \text{ where } 
\beta_{i_j,\delta_j}(x_j)=
\left(
\begin{array}{c}
 \delta_j \\ i_j
\end{array}
\right) 
x_j^{i_j} (1-x_j)^{\delta_j-i_j}.
\end{equation}
Thanks to the previous notations, these polynomials can also be written as: 
\begin{equation}
 B_{I,\delta}(x) =\left(
\begin{array}{c}
 \delta \\ I
\end{array}
\right) 
x^{I} (1_n-x)^{\delta-I}.
\end{equation}
The expression of monomials using these polynomials is:
\begin{equation}\label{monom-expr}
{x}^I=\sum_{I\le J\le \delta} \frac{\left(
\begin{array}{c}
 J \\ I
\end{array}
\right)}{\left(
\begin{array}{c}
 \delta \\ I
\end{array}
\right)}B_{J,\delta}(x),
\mbox{ for all } I\le \delta
\end{equation}

Now, we can give the general expression of a multivariate polynomial in the Bernstein basis:
$$
p(x)=\displaystyle{\sum_{I\le\delta }  b_{I,\delta} B_{I,\delta}(x) },
$$         
where Bernstein coefficients $(b_{I,\delta})_{I\le\delta}$ are given as follows:
\begin{equation}\label{eq:bernstein-coeff-formula}
b_{I,\delta}=\sum_{J\le I} \frac{\left(
\begin{array}{c}
 I \\ J
\end{array}
\right)}{\left(
\begin{array}{c}
 \delta \\ J
\end{array}
\right)}p_J.
\end{equation}
Therefore, the generalization of Lemma~\ref{lem:bernprop} will lead to the following properties:
\begin{lemma}
\label{lem:bernprop1}
For all $x=(x_1,\dots,x_n)\in U$ we have the following properties:
\begin{enumerate}
\item Unit partition: $\displaystyle{\sum_{I\le \delta} B_{I,\delta}(x) }=1.$
\item Bounds: $0 \le B_{I,\delta}(x) \le B_{I,\delta}(\frac{I}{\delta}), \text{ for all } I\le \delta.$
\item Induction: $B_{I,\delta_{r,-1}}=\frac{\delta_r-i_r}{\delta_r}B_{I,\delta}+\frac{i_r+1}{\delta_r}B_{I_{r,1},\delta},\; \forall r=1,\dots,n, \; I\le \delta_{r,-1}.$
\end{enumerate}
\end{lemma}
Also, Corollary~\ref{cor:bernbound} can be generalized as follows:
\begin{corollary}~\label{cor:bernbound1}
Let $p$ be a multivariate polynomial of degree $\delta$ over the unit box $U=[0,1]^n$ with Bernstein coefficients $b_{I,\delta}$ where $I\le \delta$. 
 Then, for all $\vx \in U$, the following inequality holds:
\begin{equation}
 \displaystyle{\min_{I \le \delta}b_{I,\delta}\le p(x) \le \max_{I\le \delta}b_{I,\delta}}.
\end{equation}

{\bf The vertex condition} holds iff the minimum value (respectively the maximum value) is reached for an index $I^* \in S_0$ where:
$$S_0=\{I=(i_1,\dots,i_n)\in \mathbb{N}^n, \mbox{ such that } i_j\in\{0,\delta_j\},\; \forall j=1,\dots,n\}.$$ 
\end{corollary}

Given the Bernstein coefficients $(b_{I,\delta})_{I \leq \delta}$ for a polynomial $p$,
the  vertex condition is quite easy to check using the steps outlined below:
\begin{enumerate}
\item Find $I^* := \argmin_{I \leq \delta} ( b_{I,\delta})$.
\item Check for each $j \in [1,n]$ if $I^*_j = 0$ or $I_j^*= \delta_j$. 
\item If the previous step succeeds, vertex condition holds and
  $b_{I^*,\delta}$ is a global minimum of $p$ inside the unit
  box. Otherwise, vertex condition fails.
\end{enumerate}
Checking the vertex condition will be an important primitive for the
overall approach that will be developed in this paper.  

Finally, consider an arbitrary, bounded interval
$\mathcal{K}:\ [\underline{x_1},\overline{x_1}]\times \dots \times
[\underline{x_n},\overline{x_n}]$,
wherein $-\infty < \underline{x_j} < \overline{x_j} < \infty $, for
all $j=1,\dots,n$. It suffices to a map $\mathcal{K}$ into the unit
box $U$ by applying the following change of variables from $x$ to $z$:
$z_j= \frac{ x - \underline{x_j}}{\overline{x_j}-\underline{x_j}}$ for
all $j=1,\dots,n$. Doing so, the results from
Lemma~\ref{lem:bernprop1} can be transferred to arbitrary boxes
$\mathcal{K}$.

\section{Bernstein Polynomial Relaxations for Polynomial Optimization Problems}\label{Sec:lp-Bernstein-relaxations}
Given a multivariate polynomial $p$ and a rectangle $\mathcal{K}$, we
consider the following optimization problem :
\begin{equation}\label{eq:POP}
\begin{array}{ll}
\text{minimize} & p(x)\\
\text{s.t} & x \in \mathcal{K} .
\end{array}
\end{equation}
Whereas~\eqref{eq:POP} is hard to solve, we will construct a linear
programming (LP) relaxation, whose optimal value is guaranteed to be a
lower bound on $p^*$.  In this section, we will use Bernstein
polynomials for the unit box ($\mathcal{K}=[0,1]^n$).  If
$\mathcal{K}$ is a general rectangle, we use an affine transformation
to transform $p$ and $\mathcal{K}$ back to the unit box.

\subsection{Reformulation Linearization Technique (RLT)} 

We first recall a simple approach to relaxing polynomial optimization
problems to linear programs, originally proposed by Sherali et
al.~\cite{sherali91,sherali97}. We then carry out these relaxations
for Bernstein polynomials, and show how the properties in
Lemma~\ref{lem:bernprop1} can be incorporated into the relaxation
schemes. Recall, once again, the optimization problem~\eqref{eq:POP}
over the unit box $\mathcal{K}$:
\[\begin{array}{lll}
p^*=&\text{minimize} & p(x)\\
&\text{s.t} & x \in \mathcal{K} =: [0,1]^n, 
  \end{array}\]
where $\mathcal{K}$ is represented by the constraints $\mathcal{K}:\ \bigwedge\limits_{j=1}^n  x_j
\geq 0\ \land\ (1- x_j) \geq 0 $.  The standard RLT approach consists
of writing $p(x) = \sum_{I} p_I x^I$ as a linear form
$p(x):\ \sum_{I} p_I y_I$ for \emph{fresh variables} $y_I$ that are
place holders for the monomials $x^I$. Next, we write down as many
\emph{facts} about $x^I$ over $\mathcal{K}$ as possible. The basic approach
now considers all possible power products up to a maximal degree $D$ i.e of the form
$\pi_{J,\delta}:\ x^J (1 - x)^{\delta-J}$ for all $J \leq \delta$ where $|\delta|=D$.
 Clearly if $x \in \mathcal{K}$ then 
$\pi_{J,\delta}(x) \geq 0$. Expanding
$\pi_{J,\delta}$ in the monomial basis as  $\pi_{J,\delta}:\ \sum_{I\leq \delta} a_{I,J} x^I$, we write the
linear inequality constraint
\[ \sum_{I \leq J} a_{I,J}\ y_I \geq 0. \]
The overall LP relaxation is obtained as
\begin{equation}\label{eq:rlt-relax-pop}
\begin{aligned}
\text{minimize}& \sum_{I} p_I y_I \\
\text{s.t.} & \sum_{I \leq J} a_{I,J}\ y_I \geq 0,\ \mbox{for each}\ J \leq \delta. \\
\end{aligned}
\end{equation}
Additionally, it is possible to augment this LP by adding inequalities
of the form $\ell_I \leq y_I \leq u_I$ through the interval evaluation
of $\vx^I$ over the set $\mathcal{K}$.

\begin{remark}
  The extra ``facts'' that form the constraints in
  Eq.~\ref{eq:rlt-relax-pop} are akin to valid inequalities or cuts
  that incrementally refine an over-approximation of the feasible
  region.  Unfortunately, the number of such inequalities is
  exponential in $|\delta|$. Rather than adding these \emph{all at
    once} to yield a single LP, we may add them \emph{on demand},
  iteratively solving a series of LPs wherein the new inequalities are
  introduced as cutting planes to help improve the solution. 
\end{remark}

\begin{proposition}
For any polynomial $p$, the optimal value computed by the LP~\eqref{eq:rlt-relax-pop} is 
a lower bound to that of the polynomial program~\eqref{eq:POP}.
\end{proposition}
\begin{example}
We wish to solve the following POP (or find a lower bound for its solution):
\begin{equation}\label{ex:rlt}
\begin{array}{ll}
\text{minimize}&  {x_1}^2+{x_2}^2  \\
\text{s.t.} & (x_1,x_2)\in[0,1]^2 \\
\end{array}
\end{equation}
Using the RLT technique for a degree $D=2$ we denote by $y_{i,j}$ the fresh variables 
replacing the non linear terms $x^{(i,j)}={x_1}^i{x_2}^j$ for all $(i,j)\in \mathbb{N}^2$ such that $i+j \le2$.
We obtain an LP which is shown, in part, below:
\[ \begin{array}{lccc}
\text{minimize} & y_{2,0} + y_{0,2} \\
\text{s.t}&\\
&0  \leq\ y_{2,0}\ \leq & 1 \\
&0  \leq\ y_{0,2}\ \leq & 1 \\
& \cdots & \\
\end{array}\]
The optimal solution obtained from the LP is $0$, which coincides with the optimum of the original problem.
\end{example}

\subsection{RLT using Bernstein Polynomials}

The success of the RLT approach depends heavily on writing ``facts''
involving the variables $y_I$ that substitute for $x^I$. We now
present the core idea of using Bernstein polynomial expansions and the
richer bounds that are known for these polynomials from
Lemma~\ref{lem:bernprop1} to improve upon the basic RLT approach.
\paragraph{Linear relaxations :}
First, we write $p(x)$ as a weighted sum of Bernstein polynomials of degree $\delta$.
\[ p(x): \sum_{I \leq \delta} b_{I,\delta} B_{I,\delta} \,, \]
wherein $b_{I,\delta}$ are calculated using the formula in
equation~\eqref{eq:bernstein-coeff-formula}.  Let us introduce a fresh
variable $z_{I,\delta}$ as a place holder for $B_{I,\delta}(x)$.
Lemma~\ref{lem:bernprop1} now gives us a set of linear inequalities
that hold between these variables $z_{I,\delta}$. 
We formulate three LP relaxations, each providing a better
approximation for the feasible region of the original
problem~\eqref{eq:POP}.

\begin{equation}
\label{eq:lpbern0}
\begin{array}{rllr}
p_{\delta}^{(0)} = & \text{minimize} & \displaystyle{\sum_{I \le \delta} b_{I,\delta} z_{I,\delta}}\\
& \text{s.t.} & z_{I,\delta} \ge 0 & I \le \delta,\\
&& \displaystyle{\sum_{I \le \delta} z_{I,\delta} =1}\\
&& z_{I,\delta}\in \mathbb{R},\; & I\le \delta. \\
\end{array}
\end{equation}

\begin{remark}
  It is easy to see that
  $p_{\delta}^{(0)}=\displaystyle{\min_{I\le \delta} b_{I,\delta}}$
  (the smallest Bernstein coefficient). As a result, it can be
  computed quite efficiently without actually invoking an LP solver.
  In fact, the branch-and-bound approach of Nataraj~\cite{Nataraj2007}
  is based on this relaxation.
\end{remark}

Using the upper bound on Bernstein polynomials from
Lemma~\ref{lem:bernprop1}, we can strengthen~\eqref{eq:lpbern0}
further, as follows:
\begin{equation}
\label{eq:lpbern1}
\begin{array}{rllrl}
{p_\delta}^{(1)}=&\text{minimize} & \displaystyle{\sum_{I\le \delta} {b}_{I,\delta} z_{I,\delta} }\\
&\text{s.t} &   z_{I,\delta} \ge 0,  & I\le \delta, \\
&   & z_{I,\delta} \le B_{I,\delta}\left( \frac{I}{\delta} \right) & I \le \delta, & \leftarrow\ \mbox{Upper Bounds}\\
&& \displaystyle{\sum_{I\le \delta} z_{I,\delta} =1} \\
&&  z_{I,\delta}\in \mathbb{R},\; & I\le \delta. \\
\end{array}
\end{equation}

Next, tighter relaxation can be obtained by adding the induction
relations between Bernstein polynomials of lower degrees.  More
precisely, using in addition the third property of
Lemma~\ref{lem:bernprop1}, we obtain the following linear program:

\begin{equation}
\label{eq:lpbern3}
\begin{array}{rllr}
{p_\delta}^{(2)}=&\text{minimize} & \displaystyle{\sum_{I\le \delta} b_{I,\delta} z_{I,\delta} }\\
&\text{s.t} & z_{I,K}\in \mathbb{R},\;  I\le K, \; K\le \delta,\\
&&  0 \le z_{I,K} \le B_{I,K}(\frac{I}{K}), \; I\le K,\; K\le \delta \\
&& \displaystyle{\sum_{I\le K} z_{I,K}  =1}, \; K\le \delta \\
&& z_{I,K_{r,-1}}=\frac{{K}_r-i_r}{{K_r}}z_{I,{K}}+\frac{i_r+1}{{K}_{r}}z_{I_{r,1},K},\; \forall r\in \{1,\dots,n\} \mbox{ s.t } {I}_{r,1}\le K. 
\end{array}
\end{equation}

\begin{remark}
Note that Eq.~\eqref{eq:lpbern1} involved variables $z_{I,\delta}$ for
$I \leq \delta$. The formulation in Eq.~\eqref{eq:lpbern3} involves a
larger set of ``lower degree'' terms of the form $z_{I,K}$ wherein
$I \leq K$ and $K \leq \delta$. These terms are, in fact, not
necessary as demonstrated in Prop.~\ref{prop:elim-terms}.
\end{remark}

Each of these relaxations provides a lower bound on the original polynomial optimization problem. 
\begin{proposition}\label{prop:lpbern-relation}
${p_{\delta}}^{(0)} \le {p_{\delta}}^{(1)} \le {p_{\delta}}^{(2)} \le p^*$ .
\end{proposition}
\begin{proof}
We already know thanks to Corollary~\ref{cor:bernbound} that ${p_{\delta}}^{(0)} \le  p^*$.

Now, consider any feasible solution $y$ to the problem~\eqref{eq:POP} which is equivalent to
 $$
 \begin{array}{llr}
 \text{minimize} & \displaystyle{\sum_{I\le \delta} b_{I,\delta} B_{I,\delta}(y)} \\
 \text{s.t} & y\in [0,1]^n. \\
 \end{array}
 $$

 We note that replacing $z_{I,\delta} = B_{I,\delta}(y)$ the vector of
 all $z_{I,\delta}$ form a feasible solution to each of the two
 relaxations. Therefore, $p_{\delta}^{(j)} \le p^*$ for all
 $j \in \{1,2\}$. Also it is easy to see that these relaxations are
 increasing (since they are constructed by adding extra constraints).
 Therefore, ${p_{\delta}}^{(0)}\le {p_{\delta}}^{(1)} \le
 {p_{\delta}}^{(2)}$ . \qed
\end{proof}

\begin{example}\label{ex2}
  Let's consider $p(x)=x^2$ on $[-1,1]$. For a degree $\delta=2$, the
  minimum of Bernstein coefficient is ${p_{\delta}}^{(0)}=-1$. Whereas
  using (\ref{eq:lpbern1}), we found ${p_{\delta}}^{(1)}=0$ which
  coincides with the optimum.

  Now, we consider the polynomial $p(x,y)=x^2+y^2$ on $[-1,1]^2$, we
  found ${p_{\delta}}^{(0)}=-2$ and ${p_{\delta}}^{(1)}=-0.5$.  Using
  (\ref{eq:lpbern3}), we obtain ${p_{\delta}}^{(2)}=0$ which is the
  exact optimal value.
\end{example}
Now, in order to simplify relaxation~\eqref{eq:lpbern3}, we formulate
an equivalent relaxation that only uses decision variables
$z_{I,\delta}$.  This is achieved by replacing lower degree variables
($z_{I,K}$ where $K\le\delta$) by a matrix product involving variables
$z_{I,\delta}$ . More precisely, we have the following result :
\begin{proposition}\label{prop:elim-terms}
There exist a matrix $A_{\delta}$ and a vector $c_\delta$ such that the LP formulation
in Eq.~\eqref{eq:lpbern3} can be written as
\begin{equation}
\label{eq:newlpbern}
\begin{array}{rllr}
{p_\delta}^{(2)}=&\text{minimize} & b_{\delta} \cdot z_{\delta} \\
&\text{s.t} &  0_{\delta} \le z_{\delta}\le u_{\delta}, \\
&&1_{\delta}\cdot z_{\delta} =1,\\
&& A_\delta z_{\delta}\le \ c_\delta,
\end{array}
\end{equation}
wherein the notation $a_{\delta}$ stands for  a vector
$(a_{I,\delta})_{I\le \delta}$, 
$u_{I,\delta}=B_{I,\delta}(\frac{I}{\delta})$, $0_{I,\delta}=0$ and
$1_{I,\delta}=1$.
\end{proposition}

\begin{proof}
Each Bernstein polynomial $B_{I,K}(x)$ can be written uniquely as
\[ B_{I,K}(x) = \sum_{J \leq \delta} \hat{b}^{(I,K)}_J B_{J,\delta} (x) \,\]
wherein $\hat{b}^{(I,K)}_J,\ J \leq \delta$ form the Bernstein coefficients
for the polynomial $B_{I,K}(x)$.
Translating this, we obtain the relation
\[ z_{I,K} = \sum_{J \leq \delta} \hat{b}^{(I,K)}_J z_{J,\delta} = \mathbf{\hat{b}}^{(I,K)} \cdot z_{\delta} \,\]
wherein $\mathbf{\hat{b}}^{(I,K)} $ is the vector of Bernstein coefficients
$(\hat{b}^{(I,K)}_{J})_{J \leq \delta}$ and $z_{\delta} = (z_{J,\delta})_{J \leq \delta}$.
The result can now be established by systematically replacing each variable $z_{I,K}$ for
$K < \delta$ in ~\eqref{eq:lpbern3} into an expression in terms of $z_{\delta}$. \qed 
\end{proof}

\begin{remark}
The computation of the pair $(A_\delta, c_ \delta)$ in
Eq.~\eqref{eq:newlpbern} depends only on $\delta$, and is independent
of the actual objective function. As a result, it can be computed
offline, once for a given problem setup in terms of number of
variables and $\delta$. 
\end{remark}

\begin{algorithm}[t]
\label{Algo1}
\SetAlgoVlined
\LinesNumbered
\DontPrintSemicolon
\KwData{POP objective $p(x)$, Constraints $g_1(x) \leq 0,\ldots, g_K(x) \leq 0$,  Box $\mathcal{B}$, Degree limit vector $\delta$.}
\KwResult{$p_{\delta}^*$: an underapproximation to the POP.}
\Begin{
\textbf{Transform}  $p, g_1, \ldots, g_k$ over the unit box $\mathcal{K}:\ [0,1]^n$ \label{nl:transform}\;
\textbf{Compute} matrices $(A_{\delta}, c_{\delta})$ \;
\textbf{Initialize}  $(\tilde{A_{\delta}},\tilde{c_{\delta}})$  empty matrices \label{nl:initialize} \;
\emph{changeOccurred} := TRUE \;
\While{changeOccurred}{
  $p_{\delta}, z_\delta := $ Solution to the LP ~\eqref{eq:newlpbern0} \;
  \emph{changeOccurred} := FALSE \;
  \For{each row $j$ in $(A_{\delta},c_{\delta})$}{
    \If{ $A_{\delta,j} z_{\delta} > c_{\delta}$ }{
      \emph{changeOccurred} := TRUE \;
      \textbf{Add} row $j$ from $A_{\delta},c_{\delta}$ to  $\tilde{A_{\delta}}, \tilde{c_{\delta}}$\;
      \textbf{Remove} row $j$  from $(A_{\delta},c_{\delta})$ \;
    }
  }\label{nl:forloop}
}
}
\caption{Overall algorithm for solving a POP using iterative Bernstein polynomial relaxation.}\label{Algo:iterative-bernstein-pop-solving}
\end{algorithm}
\paragraph{Iterative approach:}
In many cases, the optimal value given by the linear program
(\ref{eq:newlpbern}) can be obtained with fewer number of constraints
i.e instead of having the constraints given by the pair $(A_\delta,
c_\delta)$ only some of them are needed. In fact, often a large number
of constraints are inactive for the optimal solution.
More precisely, we solve LPs of the form:
\begin{equation}
\label{eq:newlpbern0}
\begin{array}{rllr}
{p_\delta}^{(2)}=&\text{minimize}& b_{\delta} \cdot z_{\delta} \\
&\text{s.t}&  0_{\delta} \le z_{\delta}\le u_{\delta}, \\
&&1_{\delta}\cdot z_{\delta} =1,\\
&& \tilde{A}_\delta z_\delta \le \tilde{c}_\delta,
\end{array}
\end{equation} 
where $({\tilde{A}}_\delta,{\tilde{c}}_\delta)$ contains a subset of
the rows in the matrix $(A_\delta, c_\delta)$.
 Algorithm~\ref{Algo:iterative-bernstein-pop-solving} shows the overall
iterative scheme.
\begin{enumerate}
\item Lines~\ref{nl:transform} to ~\ref{nl:initialize} show the initialization steps that involve computing the
matrices $(A_{\delta},c_{\delta})$. The incremental computation involves using matrices $(\tilde{A_{\delta}},\tilde{c_{\delta}})$
that are initially empty.

\item Solve the linear program \eqref{eq:newlpbern0}, which is initially the same as~\eqref{eq:lpbern1}.

\item At each step, we obtain the current optimal value ${p_\delta}$ and an optimal solution ${z_\delta}$.
\item The for loop in line~\ref{nl:forloop} iterates through all rows $j$ of the matrix $A_{\delta}$ such that  $A_{\delta,j} {z_\delta} \leq  c_{\delta,j}$ is violated. 
\item  We these violated rows  to the linear program \eqref{eq:lpbern1},  remove them from $(A_\delta, c_\delta)$.
\item Termination happens whenever no violated rows are found in the
  for loop.
\end{enumerate}

\paragraph{Exact relaxation:}
The decision variables $z_{I,\delta}$ introduced during the RLT
technique are fresh variables that substitute the nonlinear
polynomials $B_{I,\delta}(x)$. A sufficient condition for an exact
relaxation to hold is that optimal solutions
${z_{I,\delta}}^*=B_{I,\delta}({x}^*)$ for all $I\le \delta$ where
${x}^* \in \mathcal{K}$. It is easy to see that when this happens,
$x^*$ is in fact a global optimum for our problem.

\begin{proposition}
Let $z_{\delta}^*$ be the optimal value given by our relaxation. If there exist ${x}^*\in \mathcal{K} (:= [0,1]^n)$ such that
\begin{equation}\label{exact-relax}
{x^*}^I=\sum_{0\le J\le \delta} \frac{\binom{J}{I}}{\binom{\delta}{I}} {z_{J,\delta}}^*,\ 
\mbox{ for all }\ I\le \delta\,,
\end{equation}
then the relaxation is exact i.e $p^*={p_{\delta}}^*$ and ${\vx}^*$ is the global minimum. 
\end{proposition}
 \begin{proof}

   The conditions (\ref{exact-relax}) can be written as a
   $\mathcal{B}_{\delta}{z_{\delta}^*}=({{x}^*}^I)_{I\le \delta}$
   where $\mathcal{B}_{\delta}$ is the matrix given by
   (\ref{exact-relax}). If this condition holds then we have
   ${z_{I,\delta}}^*=B_{I,\delta}({x}^*)$ for all $I \le \delta$ which
   implies that
   ${p_{\delta}}^*=\displaystyle{\sum_{I\le\delta } b_{I,\delta}
     B_{I,\delta}({x}^*) }= p({x}^*)$.
   This shows that ${p_{\delta}}^*$ is also an upper bound and prove
   that the condition is sufficient. 
 \end{proof}

 The converse is not necessarily true: it is easy to construct
 examples wherein $z_{\delta}^*$ is optimal and $p_{\delta}^*$
 coincides with a global optimum, but
 $z_{\delta}^* \not= B_{\delta}(\vx^*)$ for any $\vx^*$ in the domain.
 In fact, since $z_{\delta}^*$ is not unique, then we can have
 ${z_{I,\delta}}^* \neq B_{I,\delta}({x}^*)$ whereas
 $\displaystyle{\sum_{I\le\delta } b_{I,\delta}
   {z_{I,\delta}}^*}=\displaystyle{\sum_{I\le\delta } b_{I,\delta}
   B_{I,\delta}({x}^*) }$.

 Given an optimal solution $z_{\delta}^*$, we now provide a procedure that
attempts to possibly find a $x^* \in \mathcal{K}$ such that $B_{\delta}(\vx^*) = z_{\delta}^*$:
\begin{enumerate}
\item Each variable $x_j$ is itself a polynomial and thus can be written uniquely in the Bernstein form  as $x_j : \sum_{I \leq \delta} a^{(j)}_{I} B_{I,\delta}(x)$, wherein $a^{(j)}_I$ are the Bernstein coefficients of $x_j$.
\item Therefore, compute a nominal vector $\tilde{x}$ as $\tilde{x_j}: \  \sum_{I \leq \delta} a^{(j)}_{I} z^*_{I,\delta}$.
\item Use $\tilde{x}$ to check if $z_\delta^* = B_{\delta}(\tilde{x})$. If yes, we conclude exactness of our
relaxation with $\vx^* = \tilde{\vx}$, and stop. Otherwise, we conclude that no such $\vx^*$ exists.
\end{enumerate}

\begin{example}\label{ex3}
Consider the problem in Example~\ref{ex2}. For the univariate case, one can check that relaxation (\ref{eq:lpbern1}) is exact. In fact, the optimal solution is 
$z_{\delta}^*=(0.25,0.5,0.25)$. Using the previous remark, we found $\tilde{\vx}=0.5$ and we check that ${z_{I,\delta}}^*=B_{I,\delta}(\tilde{x})$ for all $I \le \delta$. Then the relaxation is exact and ${x}^*=0.5$
which corresponds to zero after a linear transformation to $[-1,1]$.
For the bivariate case, relaxation (\ref{eq:lpbern3}) is exact but the condition ($\ref{exact-relax}$) does not hold. This is due to the fact that 
$z_{\delta}^*$ is not unique.
\end{example}

\subsection{Numerical examples}

\begin{table}[t]
  \caption{Performance of the ``monolithic'' Bernstein relaxation for benchmark problems taken from Nataraj et al~\cite{Nataraj2007}. 
    \textbf{Legend:} \textsf{ID}: problem  ID as given in~\cite{Nataraj2007}, $\delta$: the maximum degrees on variables, $p_{\delta}^{(0)}$:
    Minimum Bernstein coefficient, $p_{\delta}^{(1)}$: LP relaxation in Eq.~\eqref{eq:lpbern1}, $p_{\delta}^{(2)}$: LP relaxation in Eq.~\eqref{eq:newlpbern},
    $\mbox{\#r}(A_{\delta})$: number of constraints for formulation in Eq.~\eqref{eq:newlpbern}, $\mbox{\#r}(\tilde{A}_{\delta})$: number of rows
    for the reduced problem in Eq.~\eqref{eq:newlpbern0}, $k$: number of iterations, $t$: time (seconds) to compute the matrices $(A_{\delta}, c_{\delta})$.} \label{tab}
\begin{center}
{\footnotesize
\begin{tabular}{|c|c|c|c|c|c|c|c|c|}
\hline 
 &&&&&&&&\\
\textsf{ID} & $\delta$ & ${p_\delta}^{(0)}$ & ${p_\delta}^{(1)}$ & ${p_\delta}^{(2)}$ & \#r($A_\delta$) & \#r(${\tilde{A}}_\delta$ )& $k$ & $t$\\ 
\hline 
1 & (4,4) & -1170 & -911.47 & -856.42 & 200 & 6 & 3 & 0.1 \\
\hline 
2 & (6,4) & -7990.8  & -7195 & -6709.9 & 385 & 3 & 2 & 0.2\\
\hline 
3 & (2,2) & -926 & -451 & -316 & 27 & 4 &  2& 0.1\\
\hline 
4 &(4,2) & -9994 & -6223.4 & -4721.4 & 75 & 6 & 1 & 0.1 \\
\hline 
5 & (2,2,2)& -240 & -109.5 & -66.75 & 189 & 9 & 1 & 0.1  \\
\hline 
6 & (2,4,4) & -200299 & -199930 & -139355.28 & 1563& 28 & 2 & 1.4 \\
\hline 
7 & (1,2,1) & -36.7127 & -36.7127 & -36.7127 & 42 & 0 & 1 & 0.1 \\
\hline 
8 & (2,4,4)& -20218 & -19948 & -14290.38 & 1692 & 35 & 3 & 1.4 \\
\hline 
9 & (1,1,3,3) & -3.77 & -3.77 & -3.53 &836 & 4 & 1 & 0.8 \\
\hline   
10& (1,2,2,2) & -25.2 & -21.35 & -21.35 &594 & 0 & 0 & 0.5\\
\hline 
11& (2,2,2,2) & -1020 & -542 & -260 &2700 & 36 & 1 & 1.7\\
\hline 
12& (1,1,1,1,2) & -55 & -50 & -32.5 &438 & 8 & 2 & 0.4\\
\hline 
13& (2,\ldots,2) & -11 & -6.58&-0.5&45927&449&2&1630 \\ 
\hline
14& (1,2,2,3,1,1) & -1.44 & -1.44&-1.44&9432&0&0&64.8 \\ 
\hline
15& (2,\ldots,2) & -13 & -7.5&-&-&-&-&TO \\ 
\hline
16& (2,\ldots,2) & -2.04 & -2.02&-&-&-&-&TO \\ 
\hline

\end{tabular}
}
\end{center}

\end{table}

Thus far, we have presented three LP relaxations using Bernstein
polynomials. For the formulation in Eq.~\eqref{eq:lpbern3}, we provide
a technique to reduce the number of variables by computing matrices
$(A_{\delta}, c_{\delta})$ that substitute constraints over variables
$z_{I,K}$ for $K < \delta$ in terms of variables $z_{\delta}$
(Eq.~\eqref{eq:newlpbern}).  Next, we provide an iterative approach
that avoids an upfront solution to Eq.~\eqref{eq:newlpbern},
considering an iterative and incremental addition of constraints
as in Eq.~\eqref{eq:newlpbern0}.  Also, our approach thus far is
monolithic: we translate a single instance of a POP into a LP formulation
without considering subdivisions of the feasible region $\mathcal{K}$.

We evaluate these techniques using benchmark examples proposed by Nataraj et al.~\cite{Nataraj2007}. 
In Table~\ref{tab}, we report the optimal values of the proposed 
relaxations, the size of matrices $A_\delta$ , ${\tilde{A}}_\delta$, the number of
iteration needed and the computation time for the matrix $A_\delta$ (we print `TO' if the computation time exceeds $30$ minutes). We find that 
 considerable reduction is made by considering ${\tilde{A}}_\delta$ instead of $A_\delta$
and also a considerable improvement in the lower bound is obtained when transitioning
from the simple formulation in ~\eqref{eq:lpbern1} to the larger formulation
in ~\eqref{eq:lpbern3}. However, we find that, in many cases, a monolithic LP relaxation 
by itself is not able to provide tight bounds on the optimal value.


\begin{example}\label{Himmilbeau}
  Let's consider the Himmilbeau function taken from
  ~\cite{Nataraj2007}, shown as example ID 1 in Table~\ref{tab}. The POP is given by
\begin{equation}
p(x_1,x_2)=({x_1}^2+x_2-11)^2+(x_1+{x_2}^2-7)^2 \text{ on } [-5,5]^2.
\end{equation}

Solving the LP formulation ~\eqref{eq:lpbern3} yields
${p_\delta}^{(2)}=-856.42$ .  If one used relaxation
\eqref{eq:lpbern3}), then we have a linear program with $324$
variables and $341$ constraints, without counting the roughly $628$
bounds constraints on our variables. Instead, we can solve the linear
program given by \eqref{eq:newlpbern}. In that case, we only have $25$
decision variables. The matrix $\mathcal{A}_\delta$ will contain $200$
rows and $25$ columns.  Using the iterative approach, however, we just
need $3$ iterations to obtain ${p_\delta}^{(2)}$ where the matrix
$\tilde{A}_\delta$ contains $6$ rows, in all. Thus, we achieve a
significant reduction in the size of the LP and hence the cost of
solving it.   

However, in spite of these improvements, the objective value when
using~\eqref{eq:newlpbern0} is
$-856.416$. 
This is a very coarse lower bound on the actual optimal value which is
$p^*=0$. One reason for getting a poor bound is that the considered
box is relatively big and that the optimal solution $x^*$ is located
quite far from the edges.
\end{example}
This motivates the Branch and Bound algorithms we are going to present in the next section. 
Before doing that, we will briefly show how one can extend the previous relaxations in the case of non rectangular domains.
\subsection{Extension to polyhedral and semi algebraic sets}
If $\mathcal{K}$ is a bounded polyhedral set, our POP can be formulated as follows :
\begin{equation}
\label{eq:POP+polyhedral}
\begin{array}{llr}
\text{minimize} & p(x) \\
\text{s.t} & x\in [0,1]^n, \\
&  A_0x\le b_0,
\end{array}
\end{equation}
where $A_0\in \mathbb{R}^{m\times n}$ and $b_0 \in  \mathbb{R}^m$. In fact, it suffices to compute a bounding box for the polyhedral set $\mathcal{K}$ and then map the problem to the unit
box. 


\begin{proposition}
Using the same notation, we build the following LP:
\begin{equation}
\label{eq:lpbern+polyhedral}
\begin{array}{lll}
{p_\delta}^*=&\text{minimize}&  b_{\delta} \cdot  z_{\delta} \\
&\text{s.t} &  0_{\delta}\le  z_{\delta}\le  u_{\delta}, \\
&& 1_\delta \cdot z_{\delta}=1,\\
&& \tilde{A}_\delta  z_{\delta} \le \tilde{c}_\delta,\\
&& \displaystyle{ \sum_{I\le \delta} ( A_0 \frac{I}{\delta}) z_{I,\delta}}\le b_0.
\end{array}
\end{equation}
Then ${p_\delta}^* \le p^*$, where $p^*$ is the optimal value of (\ref{eq:POP+polyhedral}).
\end{proposition}
\begin{proof}
The proof follows directly from the following property :
$$\forall x\in [0,1]^n, \; \sum_{I\le \delta} \textstyle{\frac{I}{\delta}} B_{\delta,I}(x) =x.$$
\end{proof}

Now, If $\mathcal{K}$ is a bounded semi-algebraic set, our POP can be formulated as follows :
\begin{equation}
\label{eq:POP+semi-alg}
\begin{array}{llr}
\text{minimise} & p(x) \\
\text{s.c} & x\in [0,1]^n, \\
&  g_i(x)\le 0,\;  \forall i=1,\dots,m.
\end{array}
\end{equation}
where $p$ and $g_i$ are multivariate polynomials of degree less than $\delta$ for all $i=1,\dots,m$. Then, we have the following result :
\begin{proposition}
Recall LP~\eqref{eq:lpbern+polyhedral} below:
\begin{equation}
\label{eq:lpbern+polyhedral1}
\begin{array}{lll}
{p_\delta}^*=&\text{minimize}& b_{\delta}(p) \cdot z_{\delta} \\
&\text{s.t}  &0_\delta \le z_{\delta}\le u_{\delta}, \\
&&1_\delta \cdot z_{\delta}=1 ,\\
&& \tilde{A}_\delta z_{\delta} \le \tilde{c}_\delta.\\
&& b_{\delta}(g_i) \cdot z_{\delta}\le 0, \; \forall i=1,\dots,m.
\end{array}
\end{equation}
where $b_{\delta}(p)=(b_{I,\delta}(p))_{I\le\delta}$ and $b_{\delta}(g_i)=(b_{I,\delta}(g_i))_{I\le\delta}$ are Bernstein coefficients of respectively $p$ and $g_i$ 
for all $i=1,\dots,m$.\\
Then ${p_\delta}^* \le p^*$, where $p^*$ is the optimal value of (\ref{eq:POP+semi-alg}).
\end{proposition}
\begin{proof}
It suffices to write polynomials $g_i$, for all $i=1,\dots,m$, in the Bernstein basis up to the degree $\delta$ and replace Bernstein polynomials using 
fresh variables $z_{I,\delta}$ for all $I \le \delta$.

\end{proof}


\section{Precision Improvements}\label{Sec:improvement}
We will now consider three different approaches to improving our
relaxation using the improved LP formulations proposed in this
section:
\begin{enumerate}[(a)]
\item We will show how further properties of Bernstein polynomials can
  result in \emph{multiaffine constraint system} that can be converted
  back into a LP through dualization. However, we will see that doing
  so yields \emph{impractically large LPs}. Therefore, this approach
  is of theoretical interest.

\item Next, we will consider using higher degrees $\delta$ in our LP
  formulations beyond the degrees of the original POP. However, we
  observe that the convergence is linear in $\frac{1}{\delta}$, and
  thus quite poor when compared to the growth in running times.

\item Finally, we will use a \emph{branch-and-bound} scheme that
  decomposes our problem domain into multiple smaller boxes, using
  many pruning ideas to limit the number of branches needed. In this
  context, we examine whether the improved LP relaxations can
  translate into fewer decompositions
  of the feasible region.
\end{enumerate}

\subsection{Further Valid Inequalities}

We now consider techniques for adding further valid inequality
constraints to the overall problem. As before, our goal is to ensure
that the added constraints are affine, or can somehow be converted to
an affine system of constraints.

\paragraph{Adding Known Positive Polynomials:} One simple approach,
following recent developments in so-called diagonally dominant
sum-of-squares is to add polynomials that are easy to show nonnegative
such as $D_{ij}(x):\ (x_i - x_j)^{2d} \geq 0$ and
$E_{ij}(x):\ (x_i+x_j)^{2d} \geq 0$, for pairs $x_i, x_j$ to the
system of constraints for degrees
$d \leq \frac{1}{2}\min(\delta_i,
\delta_j)$~\cite{AliAhmadi+Majumdar/2014/DSOS}.
To add such polynomials, we convert $D_{ij}$ and $E_{ij}$ to the
Bernstein basis, perform RLT by replacing Bernstein polynomials
$B_{I,\delta}(x)$ with a fresh variable $z_{I,\delta}$.  The resulting
constraints will also be added to the matrix $(A_{\delta},c_{\delta})$
and possibly included in the matrix
$(\tilde{A}_{\delta}, \tilde{c}_{\delta})$. However, the cone of
positive polynomials over $\mathcal{K}$ is not finitely generated cone
(even when we consider positive polynomials of bounded degrees).
Therefore, an addition of finitely many generators cannot be useful
for all problems, in general.

\subsubsection{Adding Multiaffine Constraints}
In this section, we briefly sketch a further approach to LP
relaxations that involves adding multiaffine constraints and relaxing
the resulting set of constraints back to a linear program.  The
multiaffine constraints are given by product of Bernstein polynomials.
Consider Bernstein polynomials $p_1(x):= B_{I_1,\delta_1}(x), p_2:= B_{I_2,\delta_2},\ \ldots\, p_j:= B_{I_j,\delta_j}$.
\begin{claim}
  The product $p_1 \times p_2 \times \cdots \times p_j$ is of the form
  $c(I_1,\ldots,I_j,\delta_1,\ldots,\delta_j) \times B_{I_1 + \cdots +
    I_j, \delta_1+\cdots + \delta_j} $,
  where $c(I_1,\ldots,I_j,\delta_1,\ldots,\delta_j)$ is a constant
  coefficient given by the ratio of the binomial coefficients.
\end{claim}

This allows us to provide additional constraints in the formulation~\eqref{eq:lpbern3}
of the form:
\[ z_{I_1,K_1} z_{I_2,K_2} \cdots z_{I_j,K_j} = c(I_1,\ldots,I_j, \delta_1, \ldots, \delta_j)z_{I_1+\ldots,I_j,K_1+\ldots+K_j} \]

The addition of these constraints yields a system of linear multiaffine constraints of the following form:
\begin{equation}\label{eq:bernmultiaffine}
\begin{array}{rcccc}
\min & c\dot z_{\delta} \\
s.t. & \tilde{A_{\delta}} z_{\delta} & \leq & \tilde{c_{\delta}} & \mbox{Linear relationships between Bernstein polynomials}\\
     & z_{\delta} & \in & [L_{\delta},U_{\delta}] & \mbox{bounds constraints } \\
     & z_{i_1} \cdots z_{i_j} &=& c_i z_i& \mbox{ multiaffine equality constraints } \\
\end{array}
\end{equation}

As such, the multi-affine system above is, in fact, a nonlinear system of constraints. However, the following result
by Ben Sassi and Girard~\cite{Bensassi+Automatica}, shows that any such system can be relaxed to yield a linear
programs.

\begin{claim}[Ben Sassi + Girard~\cite{Bensassi+Automatica}].
  The multi-affine formulation in Eq.~\eqref{eq:bernmultiaffine} can be
  relaxed to yield a linear program whose optimal value
  lower bounds that of the multi-affine
  system~\eqref{eq:bernmultiaffine}.
\end{claim}

The central idea behind Ben Sassi and Girard's result involves writing
down the Lagrangian $L(z,\mu,\lambda)$ involving the primal variables
$z$ and multipliers $\mu, \lambda$ for the equalities and inequalities
in the optimization problem~\eqref{eq:bernmultiaffine}. It is noted that the
function $L$ is multi-affine in $z$, and also that the optimal value of
a multi-affine function in a box $[L_{\delta}, U_{\delta}]$ is achieved
at its vertices. Therefore, the dual is obtained as
$\min_z L(\mu,\lambda) = \min_{v \in V} L(v,\mu,\lambda)$, where $V$
represents the verticesof the box $[L_{\delta},U_{\delta}]$.  As a
result of this, the resulting LP is exponential in the number of
variables in $z_{\delta}$, which is already $O(n^{|\delta|})$.

As a result, even the addition of additional multi-affine facts
involving Bernstein polynomials can cause an unacceptable blowup in
the problem size.

\subsection{Higher Degree Relaxations}

To improve the precision of the computed lower bound one can increase
the degree of the relaxation $\delta$. However, if we use the simpler
formulations in Eq.~\eqref{eq:lpbern1}, then increasing
$\delta$ alone does not necessarily yield a better optimal value.

\begin{example}
  In the Example~\ref{ex2}, we saw that for $\delta=(2,2)$, the
  optimal value ${p_{\delta}}^{(1)}=-0.5$ using formulation in
  Eq.~\eqref{eq:lpbern1}. Now increasing the degree to
  $\delta'=(3,2)$, one can verify that ${p_{\delta'}}^{(1)}=-0.59$
  which is a worse bound.
\end{example}

However, if we used the formulation in Eq.~\eqref{eq:lpbern3}
or the equivalent formulations in ~\eqref{eq:newlpbern} and
~\eqref{eq:newlpbern0}, then it is easy to see that increasing
the degree $\delta$ will result in the addition of more constraints
to the LP and thus, cannot make the lower bound worse.
Increasing the degree of the approximation eventually results in 
tighter bounds that asymptotically converge to the globally optimal bound.
This  is motivated by the following
result by Lin and Rokne~\cite{bern3}:
\begin{proposition}
For a degree $K \in \mathbb{N}^n$, let $b_{I,K}$ denote Bernstein coefficients for a polynomial $p$. Then:
\label{theo:conv}
$$
\| b_{I,K} - p \left( {\frac{I}{K}}\right) \| = O\left(\frac{1}{k_1}+\dots+\frac{1}{k_n}\right).
$$
\end{proposition}

Nevertheless, this convergence can be quite slow in practice.

\begin{table}[t]
\caption{Improvement of the lower bound by considering higher dimension relaxations}
\label{table1}
\begin{center}
\begin{tabular}{|c|c|c|c|c|c|c|c|}\hline $\delta '$& (5,4) & (5,5) & (6,6) & (10,10) & (20,19)& (20,20) 
                                                 \\\hline  $p_{\delta'}^{(2)}$& -738.918 & -582.783 & -436.57 & -165.89 &-63.89 &-62.23 \\\hline
\end{tabular} 
\end{center}

\end{table}

\begin{example}

Let's consider again the POP~\eqref{ex:rlt}, by increasing the degree, we obtain the results reported in Table~\ref{table1}.
The results show initially large improvements upon increasing the degree. However, it is clear that large
degrees are needed to approach the optimal value of $p^* = 0$.
\end{example}

This motivates us to consider the approaches developed in the previous section
inside a branch-and-bound solver that recursively partitions the
feasible region into smaller region, while lower bounding the optimal
value inside each region using the approach considered here. In this
setting, a better lower bound can potentially lead to fewer branches,
and therefore a better performance.

\subsection{Branch-and-bound scheme}
In this section, we consider the branch-and-bound approach for solving
POPs and integrate the improved LP formulation in
Algorithm~\ref{Algo:iterative-bernstein-pop-solving} into our overall
branch-and-bound scheme. Our branch-and-bound scheme is built on top
of previous work by Nataraj et al.~\cite{Nataraj2007} that is based on
a simple formulation that involves finding the minimum Bernstein
coefficient inside each box decomposition considered by the
algorithm. Additionally, their approach uses properties such as the
vertex condition and a monotonicity condition (described below) to
detect leaf nodes. We augment our approach directly inside their
framework by iteratively solving LPs as described in
Algorithm~\ref{Algo:iterative-bernstein-pop-solving}. While solving a
LP is more expensive than finding the minimum Bernstein coefficient,
we show that the extra overhead is offset by our ability to consider
fewer boxes.

\begin{algorithm}[t]
\SetAlgoVlined
\LinesNumbered
\DontPrintSemicolon
\KwData{Objective: $p(x)$, constraints $g_1(x) \geq 0, \ldots, g_k(x) \geq 0$ and domain $x \in \mathcal{K}$.}
\KwResult{Lower bound $p$ to the optimal value of POP.}
\Begin{
  worklistOfBoxes := \{ $\mathcal{K}$ \} \label{nl:init-worklist}\;
  glbMin := $+ \infty $ \label{nl:init-glb-min}\;
  \While{ $\mbox{worklistOfBoxes} \not= \emptyset$}{
    $\mathcal{B}$ := pop (worklistOfBoxes )\;
    \tcc{Call Algorithm~\ref{Algo:iterative-bernstein-pop-solving} as a subroutine }
  (pB,zB) := \textbf{Compute} Bernstein lower bound for $p(x)$ on $\mathcal{B}$ \label{nl:bernstein-bounds} \;

    \tcc{Check if we have to branch further}
    \tcc{1. Is the value computed exact for $\mathcal{B}$ }
    exact := \textbf{Check} if $(pB,zB)$ is an exact solution to $\scr{B}$ \label{nl:exactness-check}\;
    \tcc{2. Monotonicity check.}
    monotone := \textbf{Check} if $\frac{\partial p}{\partial x_r}$ is sign invariant over $B$ for each $x_r$. \;
    \tcc{3. Heuristic termination condition. eg., box size is below threshold }
    terminal := \textbf{Check} if we can terminate the branch-and-bound for $\scr{B}$ \label{nl:termination-check}\;
      \If{monotone}{
          \tcc{4. Create edge subproblem $\hat{p},\hat{B}$.}
          $\hat{pB}$ := branchAndBound($\hat{p}$, $\hat{B}$) \;
          glbMin := $\min$(glbMin, $\hat{pB}$)\;
       }\ElseIf{terminal or exact}{
        glbMin := $\min$(glbMin,pB) \;
    }\Else {
      \tcc{Branch into multiple subboxes }
      $(\mathcal{B}_1,\ldots, \mathcal{B}_k )$ := splitBox ($\mathcal{B}$) \label{nl:branch} \;
      add $\mathcal{B}_1, \ldots, \mathcal{B}_k$ to worklistOfBoxes \;
    }
  }
}
\caption{Basic Branch-And-Bound scheme for solving POPs.}\label{Algo:basic-branch-and-bound-algo}
\end{algorithm}

\subsubsection{Overview of Branch and Bound Algorithm}

The main idea of the branch-and-bound (BB) algorithm is to keep
subdividing the rectangular domain into sub-boxes until a termination
condition can be  obtained. Algorithm~\ref{Algo:basic-branch-and-bound-algo} shows the
basic branch and bound scheme. It involves repeated decompositions of
the original box $\mathcal{K}$ to construct a \texttt{worklistOfBoxes}
that should become empty (ideally) in order to ensure termination. 

The algorithm's behavior and performance depends critically on three
key operations: (a) The precise relaxation used to compure the bound
for $p(x)$ in line~\ref{nl:bernstein-bounds}, (b) The exactness 
test in lines~\ref{nl:exactness-check} and termination check in line
~\ref{nl:termination-check}, and (c) The branching step in
line~\ref{nl:branch}.

\subsubsection{Exactness Test}

The exactness test is performed to infer if the current lower bound
$pB$ for $p(x)$ over a given box $B$ is in fact the optimal value.
This is achieved by testing for the vertex condition and a
monotonicity condition. The vertex condition is described in
Corollary~\ref{cor:bernbound1} (page~\pageref{cor:bernbound1}).  This
is quite easy to test once we transform the problem from the current
box $B$ to $[0,1]^n$ using the mapping $x' = T(x)$, and compute the
Bernstein coefficients of $p(T(x))$. 
\subsubsection{Monotonicity Test}
The monotonicity test (originally proposed by Nataraj et
al.~\cite{Nataraj2007}) checks whether
$0 \not\in \frac{\partial p}{\partial x_r}$ for $r=1,\dots,n$ where
$\vx \in B$. If the partial derivative $p_r$ w.r.t some $x_r$ is sign
invariant over $B$, then the global minimum of $p$ in $B$ can be obtained at
one of the bounds: $x_r = \ell_r$ or $x_r = u_r$, depending on the
sign of $p_r$.  The derivative $p_r$ is also expressed using Bernstein
polynomials, where the coefficients are computed directly from the
Bernstein coefficients of $p$. The monotonicity test is computed along
each dimension $x_r$ by computing the Bernstein coefficients of
$\frac{\partial p}{\partial x_r}$. If the polynomial is deemed
monotone along $x_r$, then depending on the sign of the partial
derivative, $x_r$ is substituted by its lower (partial derivative is
positive) or upper (partial derivative is negative) bound in $B$.  In
particular, further decomposition of $B$ is unnecessary in this
case. However, since the global minimum may lie along a facet, we
create an ``edge'' subproblem $\hat{p}$ by substituting $x_i = \ell_i$
for each monotonically increasing variable $x_i$ and $x_i = u_i$ for each
monotonically decreasing $x_i$. The resulting subproblem has strictly
fewer variables than the original problem, and is solved recursively using
the same branch-and-bound procedure.

\subsubsection{Termination Test}

The main termination test compares the current lower bound for the box
$B$ against the best upper bound $\hat{p}$ obtained by sampling
feasible points in the original feasible region $\scr{K}$. If the
lower bound $pB \leq (1 - \epsilon) \hat{p}$ (alternatively
$pB \leq - \epsilon$ when $\hat{p} = 0$), we do not subdivide the box
further. Another approach to cutting off the branch-and-bound imposes
a bound on the volumes of boxes that can be subdivided.

\subsubsection{Computing Lower Bounds}

Next, we consider the computation of lower bounds to a polynomial
$p(x)$ over a box $B$. This is a key step in our branch-and-bound
scheme. We consider the three relaxtions defined in
Eqs.~\eqref{eq:lpbern0},~\eqref{eq:lpbern1} and ~\eqref{eq:lpbern3}.
As mentioned earlier, using~\eqref{eq:lpbern0} is equivalent to
computing the minimal Bernstein coefficient as originally suggested by
Nataraj et al.~\cite{Nataraj2007}. However, the relaxtions in
~\eqref{eq:lpbern1} and ~\eqref{eq:lpbern3} involve solving linear
programming problems that are more expensive when compared to finding
the smallest Bernstein coefficient. On the other hand, the advantage
is that we obtain tighter bounds that may allow us to use fewer
decompositions. 

As a further optimization, we build a function called ``First-LP''
that attempts to provide a lower bound for ~\eqref{eq:lpbern1} directly without using a
LP solver by finding a dual feasible solution for it.
We rewrite~\eqref{eq:lpbern1} as follows:
\begin{equation}\label{eq:bern1rewritten}
\begin{array}{rll}
\min & \vb^t \vz \\
\mathsf{s.t.} & \vz & \geq 0 \\
& - \vz & \geq - \vu \\
& 1^t \vz & =  1 \\
\end{array}
\end{equation}
wherein $\vb$ is the vector of Bernstein coefficients and $\vu$ represents
the vector of upper bounds. Let us sort the Bernstein coefficents in $\vb$ and
without loss of generality we write:
\[ b_1 \leq b_2 \leq \cdots \leq b_N \,.\]
Next, let $b_l$ be an index such that $b_l \leq 0$ and $b_{l+1} > 0$.
If $b_1 \geq 0$ then $\vz^* = 0$ is an optimal solution
to~\eqref{eq:bern1rewritten}. On the other hand, if $b_N \leq 0$, then
we take $l=N$. Note that $b_1$ is the optimal value for the
relaxation~\eqref{eq:lpbern0}. Next, we choose the index
$q = \max \{ i \in [1,l-1] \ |\ \sum_{j=1}^i u_j \leq 1 \} $.

\begin{lemma}\label{Lem:lowerbound1}
The optimal value of ~\eqref{eq:bern1rewritten} is lower bounded by 
$\max (b_1, b_{q+1} + \sum_{j=1}^q b_j u_j)$.
\end{lemma}
\begin{proof}
  We first formulate the dual to~\eqref{eq:bern1rewritten}.  Let us
  use the multiplier $\vlam$ corresponding to the upper bound
  constraints $-\vz \geq \vu$ and $\mu$ corresponding to the equality
  constraint $1^t \vz = 1$.  The (simplified) dual LP is given as
  \[ \begin{array}{rlll}
       \max\ & \mu - \vu^t \vlam \\
       \mathsf{s.t.} & \mu 1 - \vlam & \leq \vb \\
                   & \vlam & \geq 0 \\
     \end{array}\]
   We set the dual solutions as $\vlam_j = - b_j$ for $j \in [1,q]$ and 
   $\vlam_j = 0$ for $j > q$. Finally we set $\mu =  b_{q+1}$. 
   We can verify  that all the
   dual constraints are satisfied. Thus our solution is dual feasible.
   We also note that it yields a dual objective value of 
   $b_{q+1} + \sum_{j=1}^q b_j u_j$ as required. In contrast, setting $\mu = b_1$ and
   $\vlam = 0$ yields another dual feasible solution. The rest follows by 
   applying the standard weak duality theorem for linear programs.
\end{proof}
It is possible to provide precise conditions under which the dual
feasible solution is in fact dual optimal, and obtain a corresponding
primal optimal solution. The advantage of using a dual lower bound in
a branch-and-bound scheme is that it provides an improved bound
over~\eqref{eq:lpbern0} but at a reduced computational cost that
involves sorting the Bernstein coefficeints and performing a linear
time scan over them to identify the indices $l,q$ which is less expensive than solving~\eqref{eq:lpbern1}. For~\eqref{eq:lpbern3}, a lower bound 
is obtained by considering the optimal value given by First-LP, construct an associate feasible solution to it, then perform an iterative approach
to improve this optimal value.

\subsubsection{Numerical Results}
The algorithms described thus far were implemented inside the
MATLAB(tm) environment using the inbuilt \texttt{linprog} function for
solving linear programs.  We compare our three algorithms using a set
of 18 benchmarks to evaluate whether the additional inequalities lead
to (a) fewer boxes being examined by our branch-and-bound scheme and
(b) overall improvement in the computation time.  The first 16
benchmarks are collected from Nataraj et al~\cite{Nataraj2007} (taken
in the same order). In addition to those, we consider two further
challenging examples:
 \begin{itemize}
 \item The 3-dimensional Motzkin example (ID=17) : $$p(x_1,x_2,x_3)={x_1}^4{x_2}^2+{x_1}^2{x_2}^4-3{x_1}^2{x_2}^2{x_3}^2+{x_3}^6, \;  R=[-0.5, 0.5]^3\,.$$
\item  The 4-dimensional algebraic example (ID=18): $$p(x_1,x_2,x_3,x_4)={x_1}^4+{x_2}^4+{x_3}^4+{x_4}^4-4x_1x_2x_3x_4-1, \; R=[-0.1, 0.1]^4\,.$$
  \end{itemize} 

  A termination test threshold $\epsilon=10^{-9}$ is fixed for
  computing the global minimum for the first $16$ benchmarks. For the
  Motzkin example ID 17, we fix $\epsilon = 10^{-5}$ and
  $\epsilon = 10^{-3}$ for example ID 18 to deal with numerical
  issues in using the MATLAB's LP solver. We expect commercial LP 
  solvers such as CPLEX to provide us with more robustness.

  Table~\ref{table2} shows the results obtained for the various
  benchmarks using the LP relaxations labeled \textsf{0}, \textsf{1}
  and \textsf{2}, respectively in column Ineq. These correspond to the
  LPs in~\eqref{eq:lpbern0}, ~\eqref{eq:lpbern1}
  and ~\eqref{eq:lpbern1} while ~\eqref{eq:lpbern0} is computed exactly 
  (since it is only given by the smallest Bernstein coefficients) whereas only lower
  bounds are computed for ~\eqref{eq:lpbern1} and ~\eqref{eq:lpbern3}
  using the results of the previous section. 
  For completeness, we also report, separately, the results over 
  the subproblems generated by the monotonicity tests.

  \paragraph{Comparing number of subdivisions:} Did the use of a
  larger LP at each step yield fewer cells? From
  Table~\ref{table2}, we observe that indeed the use of a larger LP
  formulation with more inequalities did lead to roughly a $10\%$
  reduction in the number of cells examined, especially for the
  larger instances. 

  \paragraph{Comparing total time:} Despite the reduction in the
  number of cells, the overall computation time for LP relaxation
  \textsf{2} was slightly larger. This is clearly due to the cost 
  of the iterative approach (since some LPs need to be performed). However, for relaxation \textsf{1}, the lower bound
  given by the First-LP avoid us solving LPs, which turns out to be
  advantageous. Indeed, the advantage vanishes as soon as we 
  use an LP solver for relaxation \textsf{1}, as demonstrated by 
  a single example in Table~\ref{table3}.

  \paragraph{Accuracy of Results:} Because of the adaptive nature of
  our branch-and-bound scheme, we obtain solutions that are
  consistently close to the actual global optima.

 
\begin{table}[htdp]
\centering
 \caption{Performance of the Cuts for benchmark problems taken from Nataraj et al~\cite{Nataraj2007}+ two more examples.
    \textbf{Legend:} \textsf{ID}: problem  ID as given in~\cite{Nataraj2007} + two more examples, \textsf{Ineq.}: the LP used for lower bounding \textsf{0}: LP~\eqref{eq:lpbern0}, \textsf{1}:lower bound on~\eqref{eq:lpbern1}, \textsf{2}:LP~\eqref{eq:lpbern3}, 
\textsf{Sub.} the number of subdivisions, \textsf{Time}: time taken in seconds,
\textsf{Cutoff:} number of boxes removed using the cut-off test, \textsf{Mono}: number of boxed removed using the monotonicity test, \textsf{Opt}: the  optimal value, \textsf{Sub*, Cutoff*, Time*}: Total number of subdivisions, number cutoff and time spent solving recursive subproblems.}
 \label{table2}
\begin{tabular}{|l|l|l|l|l|l|l|l|l|l|}
\hline  ID & Ineq. & Sub & Time      & Cutoff & Mono & Sub* & Cutoff* & Time*    & Opt  \\ \hline
1          & 0      & 164      & 1.2  & 55       & 62      & 5          & 7               & 0.02 & 0       \\         
1          & 1      & 155      & 1.1  & 67       & 46      & 5          & 7               & 0.02 & 0    \\            
1          & 2      & 147      & 2.5  & 47       & 61      & 5          & 7               & 0.02 & 0                \\  \hline
2          & 0      & 100      & 1.1  & 14       & 61      & 2          & 6               & 0.01 & -1.032   \\   
2          & 1      & 97       & 1.1  & 15       & 59      & 2          & 6               & 0.01 & -1.032      \\ 
2          & 2      & 97       & 3.0  & 13       & 61      & 2          & 6               & 0.01 & -1.032        \\   \hline
3          & 0      & 194      & 0.4  & 80       & 11      & 3          & 2               & 0.00 & 0                \\  
3          & 1      & 176      & 0.4  & 85       & 6       & 3          & 2               & 0.00 & 0                \\  
3          & 2      & 173      & 0.8  & 78       & 11      & 3          & 2               & 0.00 & 0                \\  \hline
4          & 0      & 1319     & 4.8  & 751      & 4       & 9          & 3               & 0.02 & 0                \\
4          & 1      & 1199     & 4.5  & 684      & 4       & 9          & 3               & 0.02 & 0                \\  
4          & 2      & 1098     & 23.2 & 625      & 4       & 9          & 3               & 0.03 & 0                \\  \hline
5          & 0      & 388      & 3  & 126      & 146     & 20         & 22              & 0.07  & -7               \\ 
5          & 1      & 371      & 2.9   & 134      & 133     & 17         & 19              & 0.06 & -7               \\   
5          & 2      & 371      & 3.4  & 122      & 145     & 17         & 19              & 0.07 & -7               \\  \hline
6          & 0      & 784      & 19.2 & 371      & 266     & 36         & 33              & 0.50  & 0                \\ 
6          & 1      & 763      & 18.8 & 371      & 254     & 34         & 31              & 0.47 & 0                \\  
6          & 2      & 707      & 33.1 & 318      & 263     & 32         & 29              & 0.66    & 0                \\ \hline
7          & 0      & 0        & 0.1  & 0        & 0       & 0          & 0               & 0        & -36.713       \\ 
7          & 1      & 0        & 0.1  & 0        & 0       & 0          & 0               & 0        & -36.713       \\  
7          & 2      & 0        & 0.1  & 0        & 0       & 0          & 0               & 0        & -36.713       \\  \hline
8          & 0      & 637      & 15.6 & 307      & 190     & 43         & 40              & 0.58 & 0                \\  
8          & 1      & 615      & 15.2 & 300      & 181     & 40         & 37              & 0.54 & 0                \\  
8          & 2      & 580      & 27.3  & 267      & 185     & 40         & 37              & 0.61 & 0                \\  \hline
9          & 0      & 3        & 0.1  & 0        & 2       & 503        & 307             & 4.25 & -3.18        \\  
9          & 1      & 3        & 0.1  & 0        & 2       & 498        & 304             & 4.25 & -3.18        \\ 
9          & 2      & 3        & 0.1  & 0        & 2       & 498        & 304             & 4.88 & -3.18        \\  \hline
10         & 0      & 1        & 0.1  & 0        & 0       & 0          & 0               & 0        & -20.8            \\  
10         & 1      & 1        & 0.1  & 0        & 0       & 0          & 0               & 0        & -20.8            \\  
10         & 2      & 1        & 0.1   & 0        & 0       & 0          & 0               & 0        & -20.8            \\  \hline
11         & 0      & 1794     & 51.5 & 1165     & 464     & 67         & 73              & 0.69 & -16              \\  
11         & 1      & 1542     & 44.4 & 1050     & 371     & 66         & 72              & 0.69 & -16              \\  
11         & 2      & 1525     & 51.6 & 989      & 415     & 66         & 72              & 0.80  & -16              \\  \hline
12         & 0      & 18       & 0.1  & 0        & 1       & 0          & 0               & 0.00 & -30.25           \\  
12         & 1      & 18       & 0.1  & 0        & 1       & 0          & 0               & 0.00 & -30.25           \\  
12         & 2      & 18       & 0.1  & 0        & 1       & 0          & 0               & 0.00 & -30.25           \\  \hline
13         & 0      & 101      & 0.1  & 6        & 0       & 0          & 0               & 0        & -0.25            \\ 
13         & 1      & 101      & 0.1  & 6        & 0       & 0          & 0               & 0        & -0.25            \\  
13         & 2      & 101      & 0.1  & 6        & 0       & 0          & 0               & 0        & -0.25            \\  \hline
14         & 0      & 0        & 0.1  & 0        & 0       & 0          & 0               & 0        & -1.44        \\  
14         & 1      & 0        & 0.1  & 0        & 0       & 0          & 0               & 0        & -1.44        \\ 
14         & 2      & 0        & 0.1  & 0        & 0       & 0          & 0               & 0        & -1.44        \\ \hline
15         & 0      & 118      & 0.1  & 7        & 0       & 0          & 0               & 0        & -0.25            \\  
15         & 1      & 118      & 0.1  & 7        & 0       & 0          & 0               & 0        & -0.25            \\  
15         & 2      & 118      & 0.1  & 7        & 0       & 0          & 0               & 0        & -0.25            \\ \hline
16         & 0      & 18       & 0.4  & 3        & 0       & 0          & 0               & 0        & -1.74        \\ 
16         & 1      & 18       & 0.4  & 3        & 0       & 0          & 0               & 0        & -1.74        \\ 
16         & 2      & 18       & 0.4  & 3        & 0       & 0          & 0               & 0        & -1.74     \\  \hline
17         & 0      & 17874    & 1161 & 5860     & 452     & 600        & 404             & 22.92 & 0                \\
17         & 1      & 16775    & 1081 & 5772     & 290     & 600        & 404             & 23.23 & 0                \\
17         & 2      & 16641    & 1431 & 5626     & 452     & 600        & 404             & 24.22  & 0                \\ \hline
18         & 0      & 12684    & 3636  & 4080     & 2422    & 2616       & 2240            & 330   & -1         \\
18         & 1      & 12033    & 3424 & 4949     & 1447    & 2480       & 2120            & 314  & -1         \\
18         & 2      & 11983    & 3967 & 3941     & 2414    & 2416       & 2062            & 333  & -1         \\ \hline

\end{tabular}
\end{table}

\begin{table}[htdp]
\centering
 \caption{Comparison between 'First-LP' and Linprog performances}
 \label{table3}
\begin{tabular}{|l|l|l|l|l|l|l|l|l|l|l|}
\hline  ID & Cut &LP& Sub & Time      & Cutoff & Mono & Sub* & Cutoff* & Time*     & Opt  \\ \hline
1          & 1     &`First-LP' & 155      & 1.11  & 67       & 46      & 5          & 7               & 0.02 & 0       \\         
1          & 1     & Linprog & 140     & 3.57   & 62      & 42        &5  & 7               & 0.07 & 0    \\     \hline       
\end{tabular}
\end{table}

 \subsubsection{Lyapunov Stability Proofs}

 A standard approach to prove stability for polynomial dynamical
 systems is to find a polynomial Lyapunov certificate which consists
 on a positive definite function decreasing along the trajectories
 inside a region of interest.  More precisely, let $V$ be a polynomial
 candidate Lyapunov function, $\dot{V}$ its derivative and $R$ the
 region of interest taken as a rectangle containing zero (the
 equilibrium point).  To verify the asymptotic stability of the
 equilibrium, we should verify that:
 $\displaystyle{\min_{x\in R} V(x)} \ge 0$ and
 $\displaystyle{\min_{x\in R} -\dot{V}(x)} \ge 0. $

 The advantage while solving POPs arising from Lyapunov function
 synthesis problems is that a global minimum is known in advance. In
 fact since usually $V(0)=\dot{V}(0)=0$, then we already know that
 zero is the global minimum of a true Lyapunov function. Therefore, 
 a good branch-and-bound decomposition scheme for this problem decomposes 
 around the equilibrium to maximize the opportunity for exact relaxations~\cite{Bensassi+Sriram}.


 To show the efficiency of the zero decomposition, we consider $9$
 Benchmarks given in our earlier work~\cite{Bensassi+Sriram}, taken in
 order.  The goal is to verify that the candidate functions are indeed
 Lyapunov functions.  In all these examples, the region of interest is
 $R=[-1,1]^n$.  We propose to check the validity of these results by
 computing ${p_V}^*$ and ${p_{\dot{V}}}^*$ which are lower bounds on
 $V$ and $-\dot{V}$ inside $R$ using the smallest Bernstein
 coefficient (${p_V}^*(0)$, ${p_{\dot{V}}}^*(0)$) and relaxation
 (\ref{eq:lpbern1}) (${p_V}^*(1)$,${p_{\dot{V}}}^*(1)$).  We report in
 Table~\ref{table4} the results we obtained where stability is said
 verified once a precision of $10^{-9}$ is reached.  In the appendix
 we give a detailed description of the Benchmarks, the Lyapunov
 function and their associated Lie derivatives.

\begin{table}[htdp]
\begin{center}\caption{Proving bounds on Lyapunov functions and their derivatives. Legend: \textsf{EX} - ID of the example taken from Ben Sassi et al.~\cite{Bensassi+Sriram}, \textsf{$p_V^*(j)$}: Lower bounds to optimal value obtained by using LP relaxation id $j$, \textsf{$p_{Vdot}^*(j)$}: Lower bounds on optimal value of Lyapunov derivative.}\label{table4}
\begin{tabular}{|c|c|c|c|c|c|}\hline EX &  ${p_V}^*(0)$ &  ${p_V}^*(1)$ & $p_{Vdot}^*(0)$  & $p_{Vdot}^*(1)$  & Stability \\\hline 1 & $-9.2\times {10}^{-12}$& 0 & $-5.4\times {10}^{-12}$ & $-3.5\times {10}^{-15}$ & \tick \\\hline 2 & -1 & -0.0625 & $-6.2\times {10}^{-12}$ & 0 & \cross \\\hline 3 & $-5.4\times {10}^{-10}$ & 0 & $-2.9\times {10}^{-10}$ & 0 & \tick \\\hline 4 & $-2.7\times {10}^{-9}$  & 0 & $-8.7\times {10}^{-10}$  & $-1\times {10}^{-14}$  & \tick \\ \hline 5 & $-1.3\times {10}^{-10}$ & 0 & $-3.7\times {10}^{-11}$ & $-1.4\times {10}^{-14}$ & \tick \\\hline 6 &  $-3.4\times {10}^{-12}$ &  $-3.5\times {10}^{-15}$ &  $-6.9\times {10}^{-13}$ &  $-4.1\times {10}^{-13}$ & \tick \\
\hline 7 & $-1.5\times {10}^{-10}$ & $-2.1\times {10}^{-14}$ & $-9.5\times {10}^{-11}$ & $-3.8\times {10}^{-11}$ & \tick \\\hline 8 & -10.9788 & -10.9788 &  $-7.9\times {10}^{-7}$ &  $-3.3\times {10}^{-9}$ & \cross \\
\hline 9 & $-9.9\times {10}^{-10}$ & $-1.7\times {10}^{-15}$ & $-3.8\times {10}^{-11}$ & $-3.3\times {10}^{-11}$ & \tick \\\hline  \end{tabular} 
\end{center}

\end{table}

\section{Conclusions}
We present a novel approach to deal with polynomial optimization problems (POPs) by relaxing them to bigger size linear programs.
The key idea is to use Bernstein polynomials in order to build LPs that can handle many of the relations between non linear terms missed 
because of the linearization process. Contrarily to the standard RLT approach, the given LPs are easily implementable 
since only a Bernstein framework is needed (coefficients, bounds and change of variable). Thanks to the properties of Bernstein polynomials ,
tighter bounds than RLT are obtained and various techniques to improve the precision of these bounds are given. We show that our 
relaxations can be used to improve the Branch and Bound scheme given by Nataraj~\cite{Nataraj2007}. The main drawback faced in the latter case was the 
extra cost of solving LPs. We already find a way to avoid this for our first linear relaxation but not for the more precise one.
This is definitely a first goal future work. Also, we manage to extend our Brand and Bound algorithms in the case of semi algebraic constraints.

\bibliographystyle{spmpsci}
\bibliography{biblyap}

\section{Appendix}

\paragraph{Benchmark \#1:} Consider the two variable polynomial ODE:
\begin{align*}
\frac{dx}{dt} &=-12.5x+2.5x^2+2.5y^2+10x^2y+2.5y^3.\\
\frac{dy}{dt} & =-y-y^2.\\
\end{align*}

\begin{lstlisting}

Lyapunov function :
2x^2 5y^2  
      
Lyapunov derivative function :
40x^3y+10x^3-50x^2+10xy^3+10xy^2-10y^3-10y^2  
\end{lstlisting}
\paragraph{Benchmark \#2:} Consider the two variable polynomial ODE:
\begin{align*}
\frac{dx}{dt} &= 6.933333x^3+4.566667x^2-21.5x.\\
\frac{dy}{dt} & = 6.933333x^3+0.4x^2y+2.066667x^2+xy^2+0.6xy-9x-y^2-y.\\
\end{align*}

\begin{lstlisting}
Lyapunov function :
5x^2-4xy+5y^2  

Lyapunov derivative function :
41.6x^4+40x^3y+37.4000x^3 -179x^2+10xy^3+10xy^2 -10y^3 -10y^2.  
 

\end{lstlisting}
\paragraph{Benchmark \#3:} Consider the two variable polynomial ODE:
\begin{align*}
\frac{dx}{dt} &=-1.5x-x^2+0.5xy+0.5y^2-2x^3+x^2y.\\
\frac{dy}{dt} & =-0.5y.\\ 
\end{align*}

\begin{lstlisting}
Lyapunov function :
5x^2+5y^2.  
    
Lyapunov derivative function :
-20x^4+10x^3y-10x^3+5x^2y-15x^2+5xy^2-5y^2.  


  \end{lstlisting}
\paragraph{Benchmark \#4:} Consider the two variable polynomial ODE:
\begin{align*}
\frac{dx}{dt} &=-2x^3-0.5xy-0.5x.\\
\frac{dy}{dt} & =0.25xy^2-0.125xy+0.25y^2-0.4125y.\\ 
\end{align*}

\begin{lstlisting}
Lyapunov function :
5x^2+5y^2.
    
Lyapunov derivative function :
-20x^4 -5x^2y -5x^2+2.5xy^3-1.25xy^2+2.5y^3-4.125y^2.  
  


  \end{lstlisting}
 \paragraph{Benchmark \#5:} Consider the three variable polynomial ODE:
\begin{align*}
\frac{dx}{dt} &=-2x^3-0.5xy-0.5x-z^3-z^2.\\
\frac{dy}{dt} & =0.25xy^2-0.125xy+0.25y^2-0.4125y.\\ 
\frac{dz}{dt} & =-z^2-z.\\
\end{align*}

\begin{lstlisting}
Lyapunov function :
5x^2+5y^2+5z^2.  

Lyapunov derivative function :
-20x^4-5x^2y-5x^2+2.5xy^3-1.25xy^2-10xz^3-10xz^2+2.5y^3-4.125y^2-10z^3-10z^2. 
 
  \end{lstlisting}
  
 \paragraph{Benchmark \#6:} Consider the three variable polynomial ODE:
\begin{align*}
\frac{dx}{dt} &=-0.5x^3y+0.5x^3z^2-3x^3+y^5-y^4+yz^4-z^4.\\
\frac{dy}{dt} & =0.25y^2-0.25y.\\ 
\frac{dz}{dt} & =yz^4+z^4-2z^3.\\
\end{align*}

\begin{lstlisting}
Lyapunov function :
1.9150x^4+5x^3+5x^2y^2+5x^2z^2+3.9396x^2-2.5409xy^3+2.5409xy^2       
+5y^4+5y^3+5y^2z^2+5y^2+5z^4+5z^3+5z^2 .

Lyapunov derivative function : 
-3.8300x^6y+3.8300x^6z^2-22.98x^6-7.5x^5y+7.5x^5z^2-45x^5-5x^4y^3+5x^4y^2z^2  
-30x^4y^2-5x^4yz^2-3.9396x^4y+5x^4z^4-26.0604x^4z^2-23.6375x^4+7.66x^3y^5       
-6.3895x^3y^4-1.2705x^3y^3z^2+6.3524x^3y^3+1.2705x^3y^2z^2-7.6228x^3y^2       
+7.66x^3yz^4-7.66x^3z^4+15x^2y^5-15x^2y^4+2.5x^2y^3-2.5x^2y^2+10x^2yz^5  
+15x^2yz^4+10x^2z^5-35x^2z^4+10xy^7-10xy^6+10xy^5z^2+7.8792xy^5-10xy^4z^2  
-9.7849xy^4+10xy^3z^4+3.1762xy^3-10xy^2z^4-1.2705xy^2+10xyz^6+7.8792xyz^4  
-10xz^6-7.8792xz^4-2.5409y^8+5.0819y^7-2.5409y^6+5y^5-2.5409y^4z^4-1.25y^4       
+10y^3z^5+5.0819y^3z^4+2.5y^3z^2-1.25y^3+10y^2z^5-22.5409y^2z^4-2.5y^2z^2  
-2.5y^2+20yz^7+15yz^6+10yz^5+20z^7-25z^6-20z^5-20z^4.  
  \end{lstlisting}
  
  \paragraph{Benchmark \#7:} Consider the three variable polynomial ODE:
\begin{align*}
\frac{dx}{dt} &=-0.5x^3y+0.5x^3z^2-x^3+y^4z+y^4-yz^3+yz^2+z^3-z^2\\
\frac{dy}{dt} & = 0.5y^2z-0.5y^2-2y\\ 
\frac{dz}{dt} & =-yz^2+yz+z^2-z\\
\end{align*}
  \begin{lstlisting}
 Lyapunov function :
-1.2500x^4+1.6667x^3+5x^2y^2+5x^2z^2+5x^2+5xy^3+5xy^2z+5xy^2+1.0921xz^3-
1.0921xz^2+5y^4+5y^3z+5y^3+5y^2z^2+5y^2z+5y^2+0.4638yz^3-0.4638yz^2
+5z^4+5z^3+5z^2.  

 Lyapunov derivative function : 
 2.5x^6y-2.5x^6z^2+5x^6-2.5x^5y+2.5x^5z^2-5x^5-5x^4y^3+5x^4y^2z^2-10x^4y^2       
-5x^4yz^2-5x^4y+5x^4z^4-5x^4z^2-10x^4-5x^3y^4z-7.5x^3y^4+2.5x^3y^3z^2-2.5x^3y^3z    
-7.5x^3y^3+2.5x^3y^2z^3+2.5x^3y^2z^2-5x^3y^2z-5x^3y^2+4.454x^3yz^3-4.454x^3yz^2  
+0.546x^3z^5-0.546x^3z^4-6.0921x^3z^3+6.0921x^3z^2+5x^2y^4z+5x^2y^4+5x^2y^3z    
-5x^2y^3-20x^2y^2-15x^2yz^3+15x^2yz^2+15x^2z^3-15x^2z^2+10xy^6z+10xy^6       
+10xy^4z^3+10xy^4z^2+17.5xy^4z+2.5xy^4-10xy^3z^3+10xy^3z^2+5xy^3z-35xy^3       
+10xy^2z^3-5xy^2z^2-25xy^2z-20xy^2-10xyz^5+6.7237xyz^4-4.5395xyz^3  
+7.8158xyz^2+10xz^5-6.7237xz^4+4.5395xz^3-7.8158xz^2+5y^7z+5y^7+5y^6z^2  
+10y^6z+5y^6+10y^5z-10y^5+1.0921y^4z^4-5y^4z^3+6.4079y^4z^2+5y^4z    
-47.5y^4-5y^3z^4+10y^3z^2-30.0000y^3z-35y^3+3.8406y^2z^4+11.8551y^2z^3  
-30.6957y^2z^2-25y^2z-20y^2-1.0921yz^6-17.8158yz^5+5.2992yz^4+1.7535yz^3  
+11.8551yz^2+1.0921z^6+17.8158z^5-3.9079z^4-5z^3-10z^2.  

   \end{lstlisting}
   
  \paragraph{Benchmark \#8:} Consider the three variable polynomial ODE:
\begin{align*}
\frac{dx}{dt} &=-0.5x^3y+0.5x^3z^2-x^3+y^4z+y^4-yz^3+3yz^2+z^3-3z^2\\
\frac{dy}{dt} & =y^4z-y^4-2y^3-z^3+3z^2\\ 
\frac{dz}{dt} & =z^2-3z\\
\end{align*}
 
 \begin{lstlisting}
Lyapunov function :
2.3519x^5-1.0449x^4y+0.3429x^4z+5x^4+4.4496x^3y^2-2.5x^3y+5x^3z^2+5x^3+5x^2y^3+5x^2y^2       
+5x^2yz^2+5x^2z^3+5x^2z^2+5x^2+3.6461xy^4+5xy^3+5xy^2z^2+5xyz^3+5xyz^2+5xz^4+5xz^3  
+2.5863xz^2-0.4325y^5+4.9487y^4z+5y^4+5y^3z^2-4.0594y^3+5y^2z^3+5y^2z^2+5y^2+5yz^4  
+5yz^3+4.5809yz^2+1.8568z^5+5z^4-0.7475z^3+5z^2. 
 Lyapunov derivative function : 
-5.8798x^7y+5.8798x^7z^2 -11.7597x^7+2.0899x^6y^2-2.0899x^6yz^2-0.6857x^6yz-5.8203x^6y
+0.6857x^6z^3+10x^6z^2 -1.3715x^6z-20x^6-6.6743x^5y^3+6.6743x^5y^2z^2-9.5986x^5y^2
-11.25x^5yz^2+7.5x^5z^4-7.5x^5z^2-15x^5+10.7148x^4y^4z+7.8046x^4y^4+5x^4y^3z^2
-12.9101x^4y^3-10x^4y^2+5x^4yz^4-16.7597x^4yz^3+20.279x^4yz^2  
-5x^4y+5x^4z^5+5x^4z^4+2.8046x^4z^3-43.071x^4z^2-1.0286x^4z-10x^4+4.7194x^3y^5z
-14.9019x^3y^5+3.1948x^3y^4z^2+18.8712x^3y^4z-1.4443x^3y^4+2.5x^3y^2z^4
+1.6797x^3y^2z^3-20.0392x^3y^2z^2+2.5x^3yz^5-1.3715x^3yz^4-36.4644x^3yz^3
+92.9436x^3yz^2+2.5x^3z^6+2.5x^3z^5-2.3357x^3z^4+23.3865x^3z^3-100.0863x^3z^2
+28.3487x^2y^6z-1.6513x^2y^6+2.5008x^2y^5z-47.5x^2y^5+20x^2y^4z^3+10x^2y^4z^2
+14.9998x^2y^4z-5x^2y^4-13.3487x^2y^3z^3  +30.046x^2y^3z^2+5.8514x^2y^2z^3
-17.5460x^2y^2z^2-15x^2yz^5+45x^2yz^4-22.5024x^2yz^3+67.5x^2yz^2+10x^2z^5  
-15x^2z^4-20x^2z^3-75x^2z^2+24.5844xy^7z-4.5844xy^7+25xy^6z-34.1687xy^6+20xy^5z^3
-30xy^5+15xy^4z^4+10xy^4z^3+15xy^4z^2+10xy^4z+10xy^4-24.5840xy^3z^3+33.7529xy^3z^2
-10xy^2z^5+30xy^2z^4+5xy^2z^3-15xy^2z^2-10xyz^6+20xyz^5+45xyz^4-45xyz^3+5xz^6+10xz^5
-60xz^4-29.8273xz^3-45.5180xz^2+1.4833y^8z+5.8088y^8+19.7947y^7z^2+5.205y^7z    
-10.6745y^7+20y^6z^3-10y^6z^2-51.7682y^6z-27.8218y^6+15y^5z^4+6.3539y^5z^3
-24.0617y^5z^2+10y^5z+14.3566y^5+10y^4z^5+10y^4z^4-12.0244y^4z^3-19.4723y^4z^2
-14.8461y^4z-20y^4-5y^3z^5-14.795y^3z^4+44.3852y^3z^3+5.8381y^3z^2-5y^2z^6
+60y^2z^4-22.8216y^2z^3-66.5347y^2z^2-5yz^7+5yz^6+42.4137yz^5-22.2410yz^4
-45.8383yz^3+2.5148yz^2+9.2838z^6-9.8460z^5-56.2589z^4+16.7274z^3-30z^2.  
 
   \end{lstlisting}

 \paragraph{Benchmark \#9:} Consider the three variable polynomial ODE:
\begin{align*}
\frac{dx}{dt} &=  0.05x^2yz+0.05x^2y-0.05x^2z-0.05x^2+0.05xyz+0.05xy-0.05xz-0.05x+0.125y^3z-0.125y^3\\
&+0.125y^2z-0.125y^2+0.2yz^5+0.2yz^4-0.2z^5-0.2z^4; \\
\frac{dy}{dt} & = 0.125y^2z-0.125y^2+0.125yz-0.125y+0.2z^5+0.2z^4\\ 
\frac{dz}{dt} & =-0.1z^2-0.1z\\
\end{align*}
 
  \begin{lstlisting}
Lyapunov function :
2.5x^2+2.5y^2+5z^2  

Lyapunov derivative function :   
0.25x^3yz+0.25x^3y-0.25x^3z-0.25x^3+0.25x^2yz+0.25x^2y-0.25x^2z-0.25x^2+0.625xy^3z
-0.6250xy^3+0.6250xy^2z-0.6250xy^2+xyz^5+xyz^4-xz^5-x+z^4+0.625y^3z-0.625y^3+0.625y^2z
-0.625y^2+yz^5+yz^4-z^3-z^2.  
   
    \end{lstlisting}

\end{document}